\documentclass{svmult} 
\usepackage{amsmath,amssymb}

\def\paper{contribution} 

\let\OldEndProof=\endproof      
\def\endproof{\mbox{}\hfill \$\OldEndProof}

\DeclareMathOperator{\MP}{MP}
\DeclareMathOperator{\id}{id}
\def\field#1{\mathbb #1}%
\def\R{\field{R}}%
\def\K{\mathcal K}%
\def\Kinf{\K_\infty}%
\providecommand{\Rn}{}                                
\renewcommand{\Rn}[1][n]{\R^{#1}}%
\newcommand{\Rnp}[1][n]{\R^{#1}_+}%
\newcommand{\Rp}{\R_+}%
\newcommand{\Rnnp}[1][n]{\R_+^{#1\times #1}}%

\begin{document}
\title{Nonlinear left and right eigenvectors for max-preserving maps}%
\author{Bj\"orn S.\ R\"uffer}
\institute{Bj\"orn S.\ R\"uffer \at School of Mathematical and Physical
  Sciences, The University of Newcastle (UON), Callaghan, NSW 2308,
  Australia, \email{bjorn.ruffer@newcastle.edu.au}}
\maketitle{}

\abstract{%
  It is shown that max-preserving maps (or join-morphisms) on the
  positive orthant in Euclidean $n$-space endowed with the
  component-wise partial order give rise to a semiring. This semiring
  admits a closure operation for maps that generate stable dynamical
  systems. For these monotone maps, the closure is used to define
  suitable notions of left and right eigenvectors that are
  characterized by inequalities. Some explicit examples are given and
  applications in the construction of Lyapunov functions are
  described.%
}%

\section{Introduction}
\label{sec:introduction}

Classical Perron-Frobenius theory asserts the existence of nonnegative
left and right eigenvectors corresponding to the dominant eigenvalue
of a nonnegative matrix \cite{%
  frobenius1908-uber-matrizen-aus-positiven-elementen.,
  frobenius1909-uber-matrizen-aus-positiven-elementen.-ii.,
  frobenius1912-uber-matrizen-aus-nicht-negativen-elementen.,
  perron1907-grundlagen-fur-eine-theorie-des-jacobischen-kettenbruchalgorithmus,
  perron1907-zur-theorie-der-matrices%
}. For (nonlinear) monotone mappings from a positive cone into itself,
various extensions to this theory have been developed, see
\cite{lemmensnussbaum2012-nonlinear-perron-frobenius-theory} and the
references therein.
While most of the nonlinear extensions consider some form of right
eigenvalue problem for monotone cone mappings, the question of left
eigenvectors has not found a lot of attention. One reason that left
eigenvectors do not have obvious counterparts in the world of
nonlinear mappings may be that they are naturally elements of the
(linear) dual of the underlying vector space in the classical spectral
theory of linear operators. Linear duals are not very natural places
to look for nonlinear eigenvectors.

In this \paper\ we consider a class of monotone mappings defined on the
positive cone in $\Rn$ equipped with the component-wise partial
order. It admits a suitable notion of left eigenvectors. This class
consists of \emph{max-preserving} mappings from $\Rnp$ into itself,
i.e., continuous, monotone maps $A\colon \Rnp\to\Rnp$ with $A0=0$ for
which $\max\{Ax,Ay\}=A\max\{x,y\}$. Instead of a numerical maximal
eigenvalue, we consider the case when a nonlinear extension of the
spectral radius is less than one, which can be characterised by the
requirement that $A^{k}x\to 0$ as $k\to\infty$ for any $x\in\Rnp$, or
alternatively by the inequality
$$
Ax\ngeq x\text{ for all }x\in\Rnp, x\ne0.
$$
Given this starting point, it is not surprising that our nonlinear
left and right ``eigenvectors'' are characterised by inequalities
rather than equations. The terms ``sub-eigenvectors'' and spectral
inequalities have been suggested as alternative terms for the objects
introduced here.  Both are (nonlinear) functions $l\colon \Rnp\to\Rp$
and $r\colon \Rp\to\Rnp$ that are continuous, zero at zero, monotone
and unbounded in every component. They satisfy
$$
l(A x) < l(x)
$$
for all $x\in\Rnp$, $x>0$ as well as
$$
A\big(r(t)\big)< r(t)
$$
for all $t>0$.

Both, $l$ and $r$ are defined via the closure $A^{*}$ of $A$ in the
semiring of max-preserving maps on $\Rnp$.

This \paper\ is organised as follows. The next section provides a little
more background on our interest in left eigenvectors.  In
Section~\ref{sec:preliminaries} we recall some necessary notation and
preliminary results. Section~\ref{sec:main-results} contains our main
results with formulas for left and right eigenvectors in
Theorems~\ref{thm:left-eigenvector} and~\ref{thm:right-eigenvector},
respectively. 
Two explicit examples are given in Section~\ref{sec:examples}.
In Section~\ref{sec:application} we explain how these eigenvectors can
be used to construct Lyapunov functions.
Section~\ref{sec:conclusion} concludes this \paper.

\section{Motivation}
\label{sec:motivation}

Our interest in left eigenvectors is rooted in the stability analysis
of interconnected systems, where the construction of Lyapunov
functions for monotone comparison systems is of special
interest~\cite{dirritorantzerruffer2015-separable-lyapunov-functions:-constructions-and-limitations}.

For a dynamical system $x(k+1)=Ax(k)$, evolving on $\Rnp$, a Lyapunov
function $V\colon\Rnp\to\Rp$ is an energy function that decreases along
trajectories. Lyapunov functions are used to prove that trajectories
converge to zero, to prove stability, or to compute regions of
attraction. Finding Lyapunov functions, however, is notoriously
hard. Basic properties they need to satisfy are continuity, positive
definiteness, radial unboundedness (i.e., $\|x\|\to\infty$ implies
$V(x)\to\infty$) and descent along trajectories, i.e., $V(Ax)<V(x)$
whenever $x\ne 0$.

If $A\in\Rnnp$ has spectral radius less than one, one can find a
positive vector $r$ (even in the case that $A$ is merely
nonnegative~\cite[Lemma~1.1]{ruffer2010-monotone-inequalities-dynamical-systems-and-paths-in-the-positive-orthant-of-euclidean-n-space})
so that $Ar\ll r$, i.e., the image under $A$ of $r$ is less than the
vector $r$ in every component. Such a vector determines a Lyapunov
function via $V(x)=\max_{i} x_{i}/r_{i}$, and this Lyapunov function
is called max-separable.

Max-separable Lyapunov functions exist for various monotone but
nonlinear systems as well, but not for
all~\cite{dirritorantzerruffer2015-separable-lyapunov-functions:-constructions-and-limitations}. In
some of these nonlinear cases one can instead find sum-separable
Lyapunov functions, which are of the form
$V(x)=\sum_{i}v_{i}(x_{i})$. If again $A\in\Rnnp$ has spectral radius
less than one, i.e., in the linear case, there exists a positive
vector $l\in\Rnp$, so that $l^{T}A\ll l^{T}$. This vector, too,
determines a Lyapunov function, $V(x)=l^{T}x$, and this one is
sum-separable. For general monotone systems however, these
sum-separable Lyapunov functions are not well understood yet, although
progress has been made in some special
cases~\cite{dirritorantzerruffer2015-separable-lyapunov-functions:-constructions-and-limitations,
  itojiangdashkovskiyruffer2013-robust-stability-of-networks-of-iiss-systems:-construction-of-sum-type-lyapunov-functions}.

As left Perron eigenvectors do determine (sum-) separable Lyapunov
functions in the linear case, there is hope that a suitable notion of
left eigenvectors will also provide Lyapunov functions in more general
scenarios. It turns out, however, that while the present definition of
left-eigenvectors does yield Lyapunov functions given by explicit
formulas, these Lyapunov functions are not separable in the above
sense.

\section{Preliminaries}
\label{sec:preliminaries}

In this work we consider $\Rn$ equipped with the component-wise
partial order, which generates the positive cone
$\Rnp=[0,\infty)^{n}$. We use the following notation.
\begin{align*}
  x&\leq y  \text{ if } y-x\in\Rnp,\\
  x&< y  \text{ if } x\leq y\text{ and }x\ne y,\\
  x&\ll y  \text{ if } y-x\text{ are in the interior of }\Rnp.
\end{align*}
Note that $\max\{x,y\}$ is the component-wise maximum of the two
vectors $x,y\in\Rn$. For notational convenience we use the binary
symbol $x\oplus y$ to denote the same thing. We also write
$\bigoplus\{x_{k}\}$ to denote the component-wise supremum of a
possibly infinite set $\{x_{k}\}$ of vectors $x_{k}\in\Rn$. 

By $\|x\|=\max_{i} |x_{i}|$ we denote the maximum-norm of
$x\in\Rn$. We note that for $x,y\in\Rn$ we have
$\|x\oplus y\|\leq \max \{\|x\|,\|y\|\}$ and equality holds if
$x,y\in\Rnp$.

The vector $(1,\ldots,1)^{T}\in\Rn$ will be denoted by $\mathbf{1}$.
The standard unit vectors in $\Rn$ are denoted by
$e_{1},\ldots,e_{n}$.

In this work we will restrict our attention to continuous and monotone
mappings. A mapping $A$ is monotone if it preserves the partial order,
i.e., $Ax\leq Ay$ whenever $x\leq y$.  The set of
\emph{max-preserving} mappings from $\Rnp$ into itself is given by
\begin{multline*}
  \MP=\MP(\Rnp)= \Big\{ A\colon\Rnp\to\Rnp %
  \text{ such that }\\
  A(x\oplus y)=(Ax)\oplus (Ay) \text{ for all }x,y\in\Rnp \Big\}.
\end{multline*}
The term max-preserving map has been coined in
\cite{karafyllisjiang2011-a-vector-small-gain-theorem-for-general-non-linear-control-systems}
in the context of stability analysis of interconnected control
systems. It coincides with the notion of join-morphisms in lattice
theory~\cite{birkoff1973-lattice-theory}. It is immediate that
max-preserving mappings are also monotone.

For $A\in\MP$ we define non-decreasing functions
$a_{ij}\colon\Rp\to\Rp$, $i,j=1,\ldots,n$, by
$a_{ij}(t)=\big(A(te_{j})\big)_{i}$ for $t\in\Rp$. It is immediate that $A$
can be represented as
$$
Ax =
\begin{pmatrix}
  a_{11}(x_{1})\oplus \ldots \oplus a_{1n}(x_{n})\\
  \vdots\\
  a_{n1}(x_{1})\oplus \ldots \oplus a_{nn}(x_{n})
\end{pmatrix},
$$
so it is natural to think of $A$ as the matrix $(a_{ij})$.

We state the following observation, where $\circ$ refers to
composition. 
\begin{lemma}
  The set $\MP$ is a $(\circ,\oplus)$-semiring with identity element
  $\id_{\Rnp}$ and neutral element $0_{\Rnp}$.
\end{lemma}
\begin{proof}
  If $A,B\in \MP$ then we verify
  $$
  (A\circ B)(x\oplus y) =%
  A (Bx\oplus By) =%
  (A\circ B)x\oplus (A\circ B)y,%
  $$
  so $\MP$ is closed under composition and
  \begin{multline*}
    (A\oplus B)(x\oplus y)=%
    A(x\oplus y) \oplus B (x\oplus y)=\\%
    (Ax\oplus Ay) \oplus (Bx \oplus By)=%
    (Ax\oplus Bx) \oplus (Ay \oplus By)=\\%
    (A\oplus B)x \oplus (A\oplus B)y,%
  \end{multline*}
  so $\MP$ is closed under the maximum operation as well.

  Clearly the identity $\id_{\Rnp}$ is a member of $\MP$ and it is the
  identity element for composition. The function $0=0_{\Rnp}$, which
  sends all of $\Rnp$ to $0\in \Rnp$, is in $\MP$, and it serves as
  neutral element for the maximum operation.
\end{proof}

For convenience we will write compositions simply as products, i.e.,
$$
A^{k}=A\circ A\circ \ldots \circ A.
$$
We make the convention that $A^{0}=\id$.

We now further restrict our attention to continuous mappings
$A\in\MP(\Rnp)$ that satisfy $A0=0$. We have the following
characterisation.

\begin{theorem}[\cite{ruffer2010-monotone-inequalities-dynamical-systems-and-paths-in-the-positive-orthant-of-euclidean-n-space}]
  \label{thm:equivalences}
  Let $A\in\MP(\Rnp)$ be continuous and satisfy $A0=0$.
  Then the following are equivalent.
  \begin{enumerate}
  \item\label{item:1} For every $x\in\Rnp$, 
    \begin{equation}
      A^{k}x\longrightarrow 0\text{ as } k\to\infty.\label{eq:1}
    \end{equation}
  \item\label{item:2} For every $x\in\Rnp$, $x\ne0$, 
    \begin{equation*}
      Ax\ngeq x. 
    \end{equation*}
  \item\label{item:3} Every cycle in the matrix $A$ is a contraction,
    i.e.,
    \begin{equation*}
      \big(
      a_{i_{1}i_{2}}\circ 
      a_{i_{2}i_{3}}\circ 
      \ldots \circ 
      a_{i_{k}i_{1}}
      \big) (t) < t
    \end{equation*}
    for every $t>0$ and all finite sequences
    $(i_{1},\ldots,i_{k})\in\{1,\ldots,n\}^{k}$.
  \item\label{item:4} All minimal cycles in $A$ are contractions,
    i.e., those that do not contain shorter cycles.
  \item\label{item:5} For every $b\in\Rnp$ there is a unique maximal
    solution $x\in\Rnp$ to the inequality
    $$
    x\leq Ax\oplus b.
    $$
  \end{enumerate}
\end{theorem}

Along with an alternative construction of a right eigenvector, a
slightly weaker version of this result has been proven
in~\cite[Theorem~6.4]{ruffer2010-monotone-inequalities-dynamical-systems-and-paths-in-the-positive-orthant-of-euclidean-n-space},
where the functions $a_{ij}$ were assumed to be either strictly
increasing or zero. However, the proof is essentially the same in the
current framework and thus omitted.

\section{Main results}
\label{sec:main-results}

Our main technical ingredient for the construction of left and right
eigenvectors is the closure of max-preserving maps in the semiring
$\MP$.

\begin{lemma}
  \label{lemma:closure}
  Let $A\in\MP(\Rnp)$ be continuous and satisfy $A0=0$. Let
  any of the conditions~\ref{item:1}--\ref{item:5} of
  Theorem~\ref{thm:equivalences} hold. Then the \emph{closure of $A$},
  given by
  \begin{equation}
    A^{*}x = \bigoplus_{k=0}^{\infty} A^{k}x\label{eq:4}  
  \end{equation}
  is a continuous and max-preserving map $A^{*}\colon \Rnp\to\Rnp$ with
  $A^{*}0=0$ that satisfies
  \begin{equation}
    A^{*}=\id \oplus AA^{*}=\id \oplus A^{*}A.\label{eq:5}  
  \end{equation}
\end{lemma}
\begin{proof}
  The identities~\eqref{eq:5} follow immediately from writing
  out~\eqref{eq:4}. That $A^{*}$ is well-defined is mostly a
  consequence of~\eqref{eq:1}, once we note that \eqref{eq:1} implies
  that the supremum in~\eqref{eq:4} is a maximum that is attained
  after a finite number of iterates of $A$.

  The $(i,j)$th entry of the matrix $A^{*}$ consists of the supremum
  over all possible paths from node $j$ to node $i$ in the weighted
  graph with $n$ vertices and directed edges weighted with the
  functions $a_{ij}$. Because any path longer than $n$ edges will
  contain a cycle, which in turn is a contraction, the infinite
  supremum in the definition of $A^{*}$, cf.~\eqref{eq:4}, is in fact
  a maximum over at most $n$ powers of $A$.
  
  Thus $A^{*}$ is max-preserving.  In particular, only a finite number
  of terms $\|A^{k}x\|$ can be larger than $\|x\|$ and they depend
  continuously on $\|x\|$.
\end{proof}

\begin{remark}
  \label{rem:Astar-is-finite-maximum}
  From the proof we see that in fact
  $$
  A^{*}x= \bigoplus_{k=0}^{n-1} A^{k}x,
  $$
  a finite maximum of only $n$ vectors instead of a supremum.  This
  will be demonstrated in Section~\ref{sec:examples}.
\end{remark}

\begin{lemma}
  \label{lem:closure-monotone-along-trajectories}
  Let $A\in\MP(\Rnp)$ be continuous and satisfy $A0=0$. Let
  any of the conditions~\ref{item:1}--\ref{item:5} of
  Theorem~\ref{thm:equivalences} hold. Then
  the \emph{closure of $A$} satisfies
  \begin{equation}
    \label{eq:2}
    A\big(A^{*}(x)\big) = A^{*}\big(A(x)\big) < A^{*}(x)
  \end{equation}
  for all $x>0$.
\end{lemma}
\begin{proof}
  First we note that from the definition~\eqref{eq:4} it follows that
  $A^{*}A=AA^{*}$.
  We have $A^{*}A\leq A^{*}A\oplus \id = A^{*}$ from~\eqref{eq:5}, so
  we only need to show that equality does not hold. To this end assume
  there is an $x\in\Rnp$, $x>0$, with $A^{*}Ax=A^{*}x$. Denoting
  $z=A^{*}x$, we have
  $$
  Az=AA^{*}x=A^{*}Ax=A^{*}x=z,
  $$
  which contradicts property~\ref{item:2} of
  Theorem~\ref{thm:equivalences}, as $z>x>0$. Hence no such $x$ can
  exist, proving that indeed $AA^{*}x=A^{*}Ax<A^{*}x$ for all $x>0$.
\end{proof}

Our main result is the following. 

\begin{theorem}[left eigenvectors for max-preserving maps]
  \label{thm:left-eigenvector}
  Let $A\in\MP(\Rnp)$ be continuous and satisfy $A0=0$. Let
  any of the conditions~\ref{item:1}--\ref{item:5} of
  Theorem~\ref{thm:equivalences} hold.
  Then $l\colon\Rnp\to\Rp$ given by
  \begin{equation}
    \label{eq:6}
    x\mapsto \mathbf{1}^{T}A^{*}(x)
  \end{equation}
  is continuous, monotone, satisfies $l(0)=0$, as well as
  \begin{enumerate}
  \item\label{item:6} $l(x)\to\infty$ whenever $\|x\|\to\infty$,
  \item\label{item:7} the left eigenvector inequality
    \begin{equation*}
      l Ax\leq l x
    \end{equation*}
    for all $x\in\Rnp$, and, moreover, $l Ax < l x$
    whenever $x\ne0$.
  \end{enumerate}
\end{theorem}
\begin{proof}
  That the map $l$ is well defined, continuous, monotone, and satisfies
  $l(0)=0$ is an immediate consequence of Lemma~\ref{lemma:closure}.
  Assertion~\ref{item:6} follows from the fact that $A^{*}\geq \id$.
  Assertion~\ref{item:7} is a direct consequence of Lemma~\ref{lem:closure-monotone-along-trajectories}.
\end{proof}

\begin{remark}
  Instead of a summation of the components of $A^{*}(x)$
  in~\eqref{eq:6} we could have taken their maximum instead, at the
  expense of loosing the strict inequality in in Assertion~2 of the
  theorem. In the context of Section~\ref{sec:application}, this would
  in general give rise to a weak Lyapunov function, i.e., one that is
  merely non-increasing along trajectories.
\end{remark}

Our notion of left eigenvectors is complemented by right eigenvectors
that are given by a similar construction, which, to the best of our
knowledge, was first demonstrated
in~\cite{karafyllisjiang2011-a-vector-small-gain-theorem-for-general-non-linear-control-systems}. A
different construction is given
in~\cite{ruffer2010-monotone-inequalities-dynamical-systems-and-paths-in-the-positive-orthant-of-euclidean-n-space}.

\begin{theorem}[right eigenvectors for max-preserving
  maps~\cite{karafyllisjiang2011-a-vector-small-gain-theorem-for-general-non-linear-control-systems}]
  \label{thm:right-eigenvector}
  Let $A\in\MP(\Rnp)$ be continuous and satisfy $A0=0$. Let
  any of the conditions~\ref{item:1}--\ref{item:5} of
  Theorem~\ref{thm:equivalences} hold.
  Then $r\colon\Rnp\to\Rp$ given by
  \begin{equation}
    \label{eq:11}
    t\mapsto A^{*}(t\mathbf{1})
  \end{equation}
  is continuous, monotone, satisfies $r(0)=0$ as well as 
  \begin{enumerate}
  \item\label{item:9} $r_{i}(t)\to\infty$ when $\|t\|\to\infty$ for every $i=1,\ldots,n$,
  \item\label{item:10} the right eigenvector inequality
    \begin{equation}
      \label{eq:12}
      A\big(r(t)\big)\leq r(t)
    \end{equation}
    for all $t\geq0$, and, moreover, $A\big(r(t)\big)< r(t)$ when
    $t>0$.
  \end{enumerate}
\end{theorem}

\begin{proof}
  That $r$ is well defined, continuous, monotone and satisfies
  $r(0)=0$ follows again from Lemma~\ref{lemma:closure}.
  Assertion~\ref{item:9} is a consequence of the fact that
  $A^{*}\geq \id$, see~\eqref{eq:5}, so $r(t)\geq t\mathbf{1}$.
  Assertion~\ref{item:10} follows from
  Lemma~\ref{lem:closure-monotone-along-trajectories} applied to
  $x=t\mathbf{1}$.
\end{proof}

\begin{remark}
  In both, Theorem~\ref{thm:left-eigenvector} and
  Theorem~\ref{thm:right-eigenvector}, instead of the vector
  $\mathbf{1}$ in the definition of $l$, respectively, $r$, any strictly positive
  vector could have been taken instead.
\end{remark}

\section{Examples}
\label{sec:examples}

We demonstrate with two examples that the the left and right
eigenvectors obtained in the previous section are given by finite
expressions, cf.\ Remark~\ref{rem:Astar-is-finite-maximum}, not as
limits as the definition in~\eqref{eq:4} might suggest. The examples
are borrowed
from~\cite{rufferito2015-sum-separable-lyapunov-functions-for-networks-of-iss-systems}. To
this end we define
\begin{multline*}
  \Kinf=\big\{a\colon\Rp\to\Rp\,\big\vert\, a\text{ is continuous,}
  \text{ unbounded, }\\
  \text{ strictly increasing and satisfies } a(0)=0 \big\},
\end{multline*}
which is the set of homeomorphisms from $\Rp$ into itself.

First we consider the case $n=2$. In this case $A$ takes the form
$$
A=
\begin{pmatrix}
   a_{11}&  a_{12}\\
   a_{21}&  a_{22}
\end{pmatrix}
$$ 
with $a_{ij}\in(\Kinf\cup\{0\})$. The associated max-preserving mapping $A\colon\Rp^{2}\to\Rp^{2}$ is given by
$$
\begin{pmatrix}
  x_{1}\\
  x_{2}  
\end{pmatrix}\mapsto
\begin{pmatrix}
  a_{11}(x_{1}) \oplus  a_{12}(x_{2})\\
  a_{21}(x_{1}) \oplus  a_{22}(x_{2})
\end{pmatrix}.
$$
The conditions of Theorem~\ref{thm:equivalences} are satisfied if and only if
\begin{align}
   a_{11}&<\id \nonumber\\
   a_{22}&<\id  \nonumber\\
  \intertext{and}
   a_{12}\circ a_{21}&<\id.
                               \label{eq:15}\\
  \intertext{Note that \eqref{eq:15} holds if and only if}
   a_{21}\circ a_{12}&<\id \nonumber
\end{align}
holds. This can be seen by observing that every $\Kinf$ function has
an inverse which is again a $\Kinf$ function. 

Writing $x=(x_{1},x_{2})^{T}$ and under the above assumptions we
compute
\begin{align}
  A^{*}(x) & = \bigoplus_{k=0}^{\infty} A^{k}(x) \nonumber\\
           & = x\oplus Ax = \big(\id_{\Rp^{2}}\oplus A\big) (x)\nonumber\\
           & = %
             \begin{pmatrix}
               \id &  a_{12}\\
                a_{21} & \id
             \end{pmatrix} (x)
                             =
             \begin{pmatrix}
               a_{11}^{*} & a_{12}^{*}\\
               a_{21}^{*} & a_{22}^{*}
             \end{pmatrix} (x)
                            \label{eq:16}
\end{align}
as already
\begin{align*}
  A^{2} & =
          \begin{pmatrix}
             a_{11}^{2}\oplus  a_{12}\circ a_{21} &  a_{12}\circ a_{22}\oplus  a_{11}\circ a_{12}\\
             a_{21}\circ a_{11}\oplus  a_{22}\circ a_{21} &   a_{22}^{2}\oplus  a_{21}\circ a_{12}
          \end{pmatrix}
\end{align*}
is component-wise less than the matrix $(\id_{\Rp^{2}}\oplus\, A)$
computed above.

From~\eqref{eq:16} we obtain
$$
l(x) = x_{1}\oplus a_{12}(x_{2}) + x_{2}\oplus a_{21}(x_{1})
$$
Notably, this function is in general not smooth and neither sum- nor
max-separable.

For the case $n=3$ things are essentially the same.

Starting from 
$$
A=
\begin{pmatrix}
   a_{11} &  a_{12} &  a_{13} \\
   a_{21} &  a_{22} &  a_{23} \\
   a_{31} &  a_{32} &  a_{33}
\end{pmatrix}
$$
and under the assumption that all cycles in $A$ are contractions, we
can compute $A^{*}$ simply
\begin{align}
  \nonumber
  A^{*} & =  \id_{\Rp^{3}}\oplus A \oplus A^{2}\\
  \label{eq:14}
        & =
          \arraycolsep 1ex      
          \begin{pmatrix}
            \id & a_{12} \oplus a_{13} \circ a_{32} & a_{13} \oplus a_{12} \circ a_{23} \\
            a_{21} \oplus a_{23} \circ a_{31} & \id & a_{23} \oplus a_{21} \circ a_{13} \\
            a_{31} \oplus a_{32} \circ a_{21} & a_{32} \oplus a_{31} \circ a_{12} & \id
          \end{pmatrix},
\end{align}
where we note that the simplifications used to obtain~\eqref{eq:14} are possible
because all cycles are contractions.

From~\eqref{eq:14} we obtain
\begin{align*}
  l(x)= %
  x_1 \oplus (a_{12} \oplus a_{13} \circ a_{32})(x_2) \oplus
  (a_{13} \oplus a_{12} \circ a_{23})(x_3)\\
  + (a_{21} \oplus a_{23} \circ a_{31})(x_{1}) \oplus x_{2} \oplus
  (a_{23}
  \oplus a_{21} \circ a_{13})(x_{3})\\
  + (a_{31} \oplus a_{32} \circ a_{21})(x_1) \oplus (a_{32} \oplus
  a_{31} \circ a_{12})(x_2) \oplus x_3.
\end{align*}

\section{Application}
\label{sec:application}

Let $A\in\MP(\Rnp)$ be continuous and satisfy $A0=0$. If
$A^{k}x\to0$ for $k\to\infty$, two types of Lyapunov functions can be
defined based on the eigenvectors introduced in the previous
section. Let $l\colon\Rnp\to\Rp$ and $r\colon\Rp\to\Rnp$ denote the
left and right eigenvectors of $A$, respectively.

Under some additional regularity assumptions, or rather,
regularisation of $r$, a max-separable Lyapunov function
$V\colon\Rnp\to\Rp$ is given by
$$
V(x) = \max_{i} r^{-1}_{i}(x_{i}),
$$
where $r_{i}$ denotes the $i$th component function of $r$. We refer
the interested reader to
\cite{karafyllisjiang2011-a-vector-small-gain-theorem-for-general-non-linear-control-systems}
or
to~\cite{dirritorantzerruffer2015-separable-lyapunov-functions:-constructions-and-limitations}
and the references therein for further details.

The left eigenvector $l$ also yields a Lyapunov function
$V\colon\Rnp\to\Rp$ simply by
$$
V(x) = l(x).
$$
Theorem~\ref{thm:left-eigenvector} establishes that this is indeed a
Lyapunov function for the system $x(k+1)=A(x(k)$.

We note that this Lyapunov function is in general neither sum- nor
max-separable. However, it has the advantage that no additional
regularity has to be assumed to make the components of the eigenvector
invertible and that it can be computed directly from the problem data.

\begin{example}
  Consider the matrix
  $$
  A=
  \newcommand{\Bold}[1]{\mathbf{#1}}
  \left(
    \begin{array}{rrr}
      \frac{1}{2} & \frac{1}{3} & \frac{1}{7} \\
      2 & \frac{1}{2} & 0 \\
      0 & 3 & \frac{1}{2}
    \end{array}\right)
  $$
  where we take the entries as linear functions $t\mapsto a_{ij}t$ and
  compute $Ax$ in max algebra, making the associated map
  $A\colon\Rnp\to\Rnp$ max-preserving.

  There are five cycles in this matrix. Three of them are
  ``self-loops'' of weight $1/2$. The other two are from node 1 to 2
  with weight 2 and back to node 1 with weight $1/3$, as well as from
  1 to 2 with weight 2, from there to 3 with weight 3 and back to 1
  with weight $1/7$. All of the loop-weights (products) are less than
  one, so this matrix satisfies the equivalent conditions of
  Theorem~\ref{thm:equivalences}.
  
  A simple computation yields
  $$
  A^{*}=
  \newcommand{\Bold}[1]{\mathbf{#1}}
  \left(
    \begin{array}{rrr}
      1 & 2 & 6 \\
      \frac{3}{7} & 1 & 3 \\
      \frac{1}{7} & \frac{2}{7} & 1
    \end{array}\right).
  $$

  From here we obtain $l(x)=\max\left\{x_{1},2x_{2},6x_{3}\right\} + %
  \max\left\{ \frac{3}{7}x_{1}, x_{2}, 3x_{3} \right\}%
  +%
  \max\left\{ \frac{1}{7}x_{1}, \frac{2}{7}x_{2}, x_{3} \right\} $ %
  and we verify that for $x>0$ the expression %
  $
  l(Ax) = 
  \max\left\{\frac{6}{7}x_{1}, 2x_{2}, 6x_{3} \right\} +%
  \max\left\{\frac{3}{7}x_{1}, \frac{6}{7}x_{2}, 3x_{3} \right\} + %
  \max\left\{\frac{1}{7}x_{1}, \frac{2}{7}x_{2}, \frac{6}{7}x_{3} \right\}%
  $ is indeed smaller.  
\end{example}

\section{Conclusion}
\label{sec:conclusion}

For max-preserving maps $A$ on $\Rnp$ we have shown that left and
right eigenvectors can be defined in a natural sense based on the
closure of the map $A$, extending the classical Perron-Frobenius
theory appropriately to nonlinear dominant eigenvalues. In this work
the dominant eigenvalue was assumed to be less than the identity, but
via suitable scaling this could be extended to more general scenarios.

Our results have been presented on $\Rnp$, however, an extension to
join-morphisms acting on Banach lattices is a natural next step.

The construction of left-eigenvectors and corresponding Lyapunov
functions for general monotone systems that are not generated by
elements of a semiring remains a challenge.


\def\cprime{$'$}

\end{document}